\documentclass[a4paper,twoside, 11pt]{article}
\usepackage[margin=1.1in]{geometry}
\usepackage{amsmath,amsfonts,amsthm}
\usepackage{graphicx,cite,times}
\usepackage{enumitem,version,lipsum}
\usepackage{pdfpages,hyperref,fancyhdr}

\newtheorem{theorem}{Theorem}[section]

\newtheorem{definition}{Definition}[section]
\newtheorem{proposition}{Proposition}

\numberwithin{equation}{section}
\fancypagestyle{klkastyle}{
	\fancyhf{}
	\fancyhead[EC]{Munesh Kumari, Jagmohan Tanti and Kalika Prasad}
	\fancyhead[OC]{On some new families of k-Mersenne and generalized $k$-Gaussian Mersenne numbers}
	\fancyfoot[C]{\thepage}
	
}
\Large

\begin{document}
	\title{On some new families of $k$-Mersenne and generalized $k$-Gaussian Mersenne numbers and their polynomials}
	\author {Munesh Kumari$^{1}$\footnote{E-mail: muneshnasir94@gmail.com, }, Jagmohan Tanti$^{2}$\footnote{E-mail: jagmohan.t@gmail.com}, Kalika Prasad$^{3}$\footnote{E-mail: klkaprsd@gmail.com}
		\\\normalsize{$^{1,3}$\small Department of Mathematics, Central University of Jharkhand, India, 835205} 
		\\\normalsize{$^{2}$\small Department of Mathematics, Babasaheb Bhimrao Ambedkar University, Lucknow, India, 226025}}    
	\date{\today}
	\maketitle
	\noindent\rule{15cm}{.15pt}
	\begin{abstract}
		In this paper, we define generalized $k$-Mersenne numbers and give a formula of generalized Mersenne polynomials and further we study their properties.  Moreover, we define Gaussian Mersenne numbers and obtain some identities like Binet Formula, Cassini's identity, D'Ocagne's Identity, and generating functions. The generalized Gaussian Mersenne numbers are described and the relation with classical Mersenne numbers are explained. We also introduce a generalization of Gaussian Mersenne polynomials and establish some properties of these polynomials.
	\end{abstract}
	\noindent\rule{15cm}{.15pt}
	\section{Introduction}
	Mersenne sequence$\{M_{n}\}$\cite{catarino2016mersenne} is given by the recurrence relation,
	\begin{equation}\label{mersene}
		M_{n+2} = 3M_{n+1} -2M_{n}, \hspace{.8cm}n\geq 0,~with~~M_{0}=0,~~M_{1}=1,
	\end{equation}
	and the terms $M_{n}$ of this sequence are known as Mersenne numbers.
	Characteristic equation corresponding to above recurrence relation is 	\begin{equation}\label{char}
		\lambda^{2}-3\lambda+2\lambda=0
	\end{equation}
	And, Binet formula for the Mersenne numbers is given by 
	\begin{equation}\label{Binet}
		M_{n}=\lambda_{1}^{n}-\lambda_{2}^{n},
	\end{equation}
	where $\lambda_{1}$ and $\lambda_{2}$ are the roots of the characteristic equation (\ref{char}).
	
	Recall that, in number theory Mersenne numbers($M_n$) are sequences of integers of the form $2^n-1$ for non-negative integers n, which can be also obtained by its Binet formula.
	\\In this paper, we study and generalize one of the recursive sequences of integers satisfying a recurrence relation and we give corresponding polynomials, some well-known identities for this type of sequence.
	Some well-known recursive sequences are Fibonacci, Lucas, Horadam, Pell, Perrin, Fibonacci-Lucas, Jacobsthal, ..etc that are studied over the years and still are a great area of interest for generalization and their applications in other disciplines like cryptography, coding theory, matrix theory..etc.
	Horadam\cite{horadam1961generalized,horadam1963complex} introduced the concept of Gaussian Fibonacci numbers and complex Fibonacci numbers. Further, Jordan\cite{jordan1965gaussian} considered Gaussian Fibonacci numbers and Lucas numbers and extended classical relations which are known for Fibonacci numbers and Lucas numbers. Moreover, studies on the different Gaussian sequences like Gaussian Fibonacci, Gaussian Lucas, Gaussian Pell, Gaussian Pell-Lucas, and their polynomials can be found in the papers\cite{ozkan2020gauss,catarino2018note,halici2016some,yagmur2018gaussian}. Also, some work in the direction of generalization of the recursive sequences like Lucas, Pell, Horadam, ...etc has been done in \cite{ozkan2021new,numbers2021new,yilmaz2022new}. Kalika et.al\cite{prasad2021cryptography} discussed the generalization of higher-order Fibonacci sequences and shows their application in cryptography as a key matrix.	
	Construction of identities related to Mersenne numbers and generalized Mersenne numbers and study of their properties have been studied in papers\cite{boussayoud2018some,catarino2016mersenne,frontczak2020mersenne,chelgham2021k,uslu2017some} using generating functions and matrix methods.
	\begin{theorem}[Cassini Identity\cite{uslu2017some}]
	For $n \geq 1$,
		\begin{equation}\label{Cassini}
			M_{n}^{2}-M_{n+1}M_{n-1}=2^{n-1}.
		\end{equation}
	\end{theorem}
	\begin{theorem}[Generating function\cite{frontczak2020mersenne}]\label{generating}
	Generating function for Mersenne numbers $M_{n}$ is given by,
		\begin{equation}
			M(x)=\sum_{i=0}^{\infty}M_{i}x^{i}=\dfrac{x}{(1-3x+2x^{2})}.
		\end{equation}
	\end{theorem}
  This paper is organized as follows. In section 2, we define generalized $k$-Mersenne numbers and their polynomials and established relations between classical Mersenne numbers and generalized $k$-Mersenne numbers. In section 3, we introduce Gaussian Mersenne numbers and obtain some identities like Binet Formula, Cassini's identity, D'Ocagne's Identity, and generating functions related to these numbers. And at last, we define $k$-generalized Gaussian Mersenne numbers and discussed their properties, and obtain some identities involving Mersenne numbers.
	\begin{center}
		\section{Main Work}
	\end{center}
	\subsection{Generalized $k$-Mersenne numbers(GMN($k$))}
	\begin{definition}\label{l-mersenne}
		Let $n,k\in \mathbb{N}$, then $\exists !$ $s,r\in \mathbb{N}$ such that $n=sk+r$, $0 \le r <k$. Then generalized $k$-Mersenne numbers $M_{n}^{(k)}$ are defined as 
		\begin{equation}
			M_{n}^{(k)}=(\lambda_{1}^{s}-\lambda_{2}^{s})^{k-r}(\lambda_{1}^{s+1}-\lambda_{2}^{s+1})^{r},~~~~~n=sk+r
		\end{equation}
		where $\lambda_{1}$ and $\lambda_{2}$ are the roots of the eqn.(\ref{char}).
	\end{definition}
	Some generalized $k$-Mersenne numbers $M_{n}^{(k)}$ are shown in the following table.\\\\
	\begin{table}[h!]
		\centering
			\begin{tabular}{c c c c c c c}
				\hline 
				$M_{n}^{(k)}$ & k=1 & k=2  &  k=3 &  k=4 & k=5   \\
				\hline
				$M_{0}^{(k)}$ & 0 & 0  &  0 &  0 & 0   \\
				
				$M_{1}^{(k)}$  &  1 & 0  &  0 &  0 & 0   \\
				
				$M_{2}^{(k)}$ & 3 & 1  &  0 &  0 & 0    \\
				
				$M_{3}^{(k)}$& 7 & 3  &  1 &  0 & 0   \\
				
				$M_{4}^{(k)}$ & 15 & 9  &  3 &  1 & 0   \\
				
				$M_{5}^{(k)}$& 31 & 21  &  9  &  3 &  1  \\
				\hline	
			\end{tabular}
		\caption{List of some generalized $k$-Mersenne numbers($M_{n}^{(k)}$)} 
	\end{table}\\
	From eqn.(\ref{Binet}) and definition(\ref{l-mersenne}), the generalized $k$-Mersenne numbers and Mersenne numbers are related as 
	\begin{equation}\label{relation}
		M_{n}^{(k)}=M_{s}^{k-r}M_{s+1}^{r},~~~n=sk+r.
	\end{equation}
	If $k=1$ then $r=0$ and hence $n=s$. So, from eqn.(\ref{relation}) we have $M_{n}^{(1)}= M_{n}$.
	\\From the above derivations, we have noted the following identities showing relations between generalized $k$-Mersenne numbers and Mersenne numbers for $k=2,3$.
	\begin{enumerate}
		\item $ M_{2n}^{(2)}=M_{n}^{2}$.
		\item $ M_{2n+1}^{(2)}=M_{n}M_{n+1}$.
		\item $ M_{2n+1}^{(2)}=3M_{2n}^{(2)}-2M_{2n-1}^{(2)}$.
		\item $ M_{3n}^{(3)}=M_{n}^{3}$.
		\item $ M_{3n+1}^{(3)}=M_{n}^{2}M_{n+1}$.
		\item $ M_{3n+1}^{(3)}=3M_{3n}^{(3)}-2M_{3n-1}^{(3)}$.
		\item $ M_{3n+2}^{(3)}=M_{n}M_{n+1}^{2}$.	
	\end{enumerate} 
	 \begin{proposition}\label{obsr}
	   Let $k,s \in \mathbb{N}$ then $ M_{sk}^{(k)}=M_{s}^{k}$.
	 \end{proposition}
	\begin{proof}
	 For $n=sk$, $r=0$.
		\\So, from eqn.(\ref{relation}) we have $$M_{sk}^{(k)}=M_{s}^{k-0}M_{s+1}^{0}=M_{s}^{k}.$$
	\end{proof}
	 \begin{theorem}
		For $n,s \in \mathbb{N}$, $ M_{sn+1}^{(s)}=3M_{sn}^{(s)}-2M_{sn-1}^{(s)}$.
	\end{theorem}
	\begin{proof}
		From eqn.(\ref{mersene}) and eqn.(\ref{relation}), we have
		\begin{eqnarray}
			3M_{sn}^{(s)}-2M_{sn-1}^{(s)} 
			&=& 3M_{n}^{s}-2M_{n-1}M_{n}^{s-1} \nonumber\\
			&=& M_{n}^{s-1}(3M_{n}-2M_{n-1})\nonumber\\
			&=& M_{n}^{s-1}M_{n+1} \nonumber\\
			&=& M_{sn+1}^{(s)}. \nonumber
		\end{eqnarray}
	\end{proof}
	\begin{theorem}
		For $k,s \in \mathbb{N}$ we have, $M_{s+1}^{k}-M_{s}^{k}=M_{sk+k}^{(k)}-M_{sk}^{(k)}$.
	\end{theorem}
	\begin{proof}
		From eqn.(\ref{relation}), we have
		\begin{eqnarray}
			M_{sk+k}^{(k)}-M_{sk}^{(k)} &=& [M_{s}^{k-k}M_{s+1}^{k}]-[M_{s}^{k-0}M_{s+1}^{0}] \nonumber\\
			&=& M_{s+1}^{k}-M_{s}^{k}. \nonumber 
		\end{eqnarray}
	\end{proof}
	\begin{theorem}
		For $n, m \ge 0$ such that $n+m>1$,
		\begin{equation}
			M_{2(n+m-1)}^{(2)}-M_{n+m}M_{n+m-2}=2^{n+m-2}.
		\end{equation}
	\end{theorem}
	\begin{proof}
		By eqn.(\ref{Cassini}) and proposition(\ref{obsr}), we get
		\begin{equation}
			M_{2(n+m-1)}^{(2)}-M_{n+m}M_{n+m-2}=M_{(n+m-1)}^{2}-M_{n+m}M_{n+m-2}=2^{n+m-2}.\nonumber
		\end{equation}
	\end{proof}
	\begin{theorem}
		Let $n,k \ge 2$ then Cassini's identity for $M_{n}^{(k)}$ is,
		\\$$M_{nk+a}^{(k)}M_{nk+a-2}^{(k)}-(M_{nk+a-1}^{(k)})^{2}=
		\begin{cases}
			-2^{n-1}M_{n}^{2k-2},~~~~ a = 1 \\
			0, \hspace{2.4cm}a \ne 1\\
		\end{cases}$$.
	\end{theorem}
	\begin{proof}
		If a$\ne$ 1, then from eqn.(\ref{relation})  
		\begin{eqnarray}
			M_{nk+a}^{(k)}M_{nk+a-2}^{(k)}-(M_{nk+a-1}^{(k)})^{2}  &=&  (M_{n}^{k-a}M_{n+1}^{a})(M_{n}^{k-a+2}M_{n+1}^{a-2})-(M_{n}^{k-a+1}M_{n+1}^{a-1})^{2}\nonumber\\
			&=& M_{n}^{2k-2a+2}[M_{n+1}M_{2a-2}-(M_{n+1})^{2a-2}].\nonumber\\
			&=& 0 \nonumber
		\end{eqnarray}
	and if $a=1$, 
		\begin{eqnarray}
			M_{nk+1}^{(k)}M_{nk-1}^{(k)}-(M_{nk}^{(k)})^{2} &=& (M_{n}^{k-1}M_{n+1})(M_{n-1}M_{n}^{k-1})-(M_{n}^{k})^{2}\nonumber\\
			&=& M_{n}^{2k-2}[M_{n+1}M_{n-1}-(M_{n})^{2}]\nonumber\\
			&=& -2^{n-1}M_{n}^{2k-2}. \hspace{.5in} \text{(Using  eqn.(\ref{Cassini}))}\nonumber
		\end{eqnarray}	 
	\end{proof}
	\subsection{Mersenne Polynomial and Generalized $k$-Mersenne polynomials}
	\begin{definition}
		The \textbf{Mersenne Polynomial} $M_{n}(x)$ is defined by the recurrence relation,
		\begin{eqnarray}\label{mersenepolynomial}
			M_{n+2}(x) = 3xM_{n+1}(x) -2M_{n}(x) \hspace{.8cm}n\geq 0,~with~~M_{0}=0,~~M_{1}(x)=1.
		\end{eqnarray}
	\end{definition}
 Characteristic equation corresponding to the recurrence relation(\ref{mersenepolynomial}) is 
 \begin{equation}\label{charpoly}
 	\lambda^{2}-3x\lambda+2\lambda=0.
 \end{equation}
\begin{theorem}
	For $ n\in \mathbb{N} $, we can write Binet formula for the Mersenne polynomial as
	\begin{equation}\label{Binetpoly}
		M_{n}(x)=\dfrac{\lambda_{1}^{n}(x)-\lambda_{2}^{n}(x)}{\lambda_{1}(x)-\lambda_{2}(x)},
	\end{equation}
	where $\lambda_{1}(x)=\dfrac{3x+\sqrt{9x^{2}-8}}{2}$ and $\lambda_{2}(x)=\dfrac{3x-\sqrt{9x^{2}-8}}{2}$ are the roots of the characteristic equation(\ref{charpoly}).
\end{theorem}
\begin{proof}
	By the theory of difference equation $n^{th}$ term of Mersenne polynomial can be written as,
	\begin{equation}\label{generalpoly}
		M_{n}(x)= a\lambda_{1}^{n}(x)+b\lambda_{2}^{n}(x). 
	\end{equation}
	From eqn.(\ref{mersenepolynomial}) we have, $M_{0}(x)= a+b$ and $M_{1}(x)= a\lambda_{1}(x)+b\lambda_{2}(x)$.\vspace{.1in} \\On solving $M_{0}(x)~\&~ M_{1}(x)$ we get,
	$a=\dfrac{1}{\lambda_{1}(x)-\lambda_{2}(x)}$ and $b=\dfrac{-1}{\lambda_{1}(x)-\lambda_{2}(x)}$.\vspace{.1in}
	\\Now, using values of $a,b$ in eqn.(\ref{generalpoly}), we get 
	$$M_{n}(x)=\dfrac{\lambda_{1}^{n}(x)-\lambda_{2}^{n}(x)}{\lambda_{1}(x)-\lambda_{2}(x)}.$$
	\end{proof}
	\begin{theorem}\label{Cassinipolymial}
			For $n \geq 1$,
			\begin{equation}
				M_{n}^{2}(x)-M_{n+1}(x)M_{n-1}(x)=2^{n-1}.
			\end{equation}
	\end{theorem}
	\begin{proof}
		It can be proved using mathematical induction on $n$.
	\\For $n=1$, 
	\begin{equation}
			M_{1}^{2}(x)-M_{2}(x)M_{0}(x)=1-0.3x=2^{0}.\nonumber
	\end{equation}  
	So, result is true for $n=1$.
	\\Assume that result is true for $n=k$, $i.e.$ 
	\begin{equation}\label{Cassinipoly2}
		M_{k}^{2}(x)-M_{k+1}(x)M_{k-1}(x)=2^{k-1}. 
	\end{equation}
	Now for $n=k+1$,
	\\using eqn.(\ref{mersenepolynomial}) and eqn.(\ref{Cassinipoly2}), we have
	\begin{eqnarray}
		M_{k+1}^{2}(x)-M_{k+2}(x)M_{k}(x) &=& M_{k}^{2}(x)-\left[\left(3M_{k+1}(x)-2M_{k}(x)\right)\left(\frac{1}{3}M_{k+1}(x)+\frac{2}{3}M_{k-1}(x)\right)\right]\nonumber\\
		&=& -2M_{k+1}(x)M_{k-1}(x)+\frac{2}{3}M_{k+1}(x)M_{k}(x)+\frac{4}{3}M_{k-1}(x)M_{k}(x)\nonumber\\
		&=& 2M_{k}^{2}(x)+2.2^{k-1}-\frac{2}{3}M_{k+1}(x)M_{k}(x)+\frac{4}{3}M_{k-1}(x)M_{k}(x) \nonumber\\
		&=& 2^{k}.\nonumber
	\end{eqnarray}

\end{proof}
	\begin{definition}[Generalized $k$-Mersenne polynomial ($M_{n}^{(k)}(x)$)] \label{l-mersennepolnomial}
		Let $n,k\in \mathbb{N}$, then $\exists$ $s,r$ such that $n=sk+r$, $0 \le r <k$. The generalized $k$-Mersenne polynomial $M_{n}^{(k)}(x)$ is defined by 
		\begin{equation}
			M_{n}^{(k)}(x)=\left(\dfrac{\lambda_{1}^{s}(x)-\lambda_{2}^{s}(x)}{\lambda_{1}(x)-\lambda_{2}(x)}\right)^{k-r}\left(\dfrac{\lambda_{1}^{s+1}(x)-\lambda_{2}^{s+1}(x)}{\lambda_{1}(x)-\lambda_{2}(x)}\right)^{r},~~~~~n=sk+r
		\end{equation}
	where $\lambda_{1}(x)$ and $\lambda_{2}(x)$ are the roots of the characteristic equation (\ref{charpoly}).
	\end{definition}
	From eqn.(\ref{Binetpoly}) and definition(\ref{l-mersennepolnomial}), we have the relation between generalized $k$-Mersenne polynomials and Mersenne polynomials as
	\begin{equation}\label{polyrelation}
		M_{n}^{(k)}(x)=M_{s}^{k-r}(x)M_{s+1}^{r}(x),~~~n=sk+r.
	\end{equation}
	If $k=1$ then $r=0$ and hence $n=s$. So, from eqn.(\ref{polyrelation},) we have $M_{n}^{(1)}(x)= M_{n}(x)$.
	\\Some values of generalized $k$-Mersenne polynomials $M_{n}^{(k)}(x)$ has been shown in following table.
	\begin{table}[h]
		\centering
			\begin{tabular}{c c c c c c c}
				\hline 
				$M_{n}^{(k)}(x)$ & k=1 & k=2  &  k=3 &  k=4 & k=5   \\
				\hline
				$M_{0}^{(k)}(x)$ & 0 & 0  &  0 &  0 & 0   \\
				
				$M_{1}^{(k)}(x)$  &  1 & 0  &  0 &  0 & 0   \\
				
				$M_{2}^{(k)}(x)$ & 3$x$ & 1  &  0 &  0 & 0    \\
				
				$M_{3}^{(k)}(x)$& 9$x^{2}$-2 & 3$x$  &  1 &  0 & 0   \\
				
				$M_{4}^{(k)}(x)$ & 27$x^{3}$-12$x$ & 9$x^{2}$  &  3$x$ &  1 & 0   \\
				
				$M_{5}^{(k)}(x)$& 81$x^{4}$-54$x^{2}$+4 & 27$x^{3}$-6$x$  &  9$x^{2}$  &  3$x$ &  1 \\
				\hline	
			\end{tabular}
		\caption{Some generalized $k$-Mersenne polynomials $M_{n}^{(k)}(x)$} 
		\label{table:2}
	\end{table}
  \\From the table(\ref{table:2}) and eqn.(\ref{polyrelation}), we have the following relations between the generalized $k$-Mersenne polynomials and Mersenne polynomials for $k=2,3$,
	\begin{enumerate}
		\item $ M_{2n}^{(2)}(x)=M_{n}^{2}(x)$.
		\item $ M_{2n+1}^{(2)}(x)=M_{n}(x)M_{n+1}(x)$.
		\item $ M_{2n+1}^{(2)}(x)=3M_{2n}^{(2)}(x)-2M_{2n-1}^{(2)}(x)$.
		\item $ M_{3n}^{(3)}(x)=M_{n}^{3}(x)$.
		\item $ M_{3n+1}^{(3)}(x)=M_{n}^{2}(x)M_{n+1}(x)$.
		\item $ M_{3n+1}^{(3)}(x)=3M_{3n}^{(3)}(x)-2M_{3n-1}^{(3)}(x)$.
		\item $ M_{3n+2}^{(3)}(x)=M_{n}(x)M_{n+1}^{2}(x)$.	
	\end{enumerate}
	\begin{proposition}\label{obsr1}
		For $k,n \in \mathbb{N}$, we have $ M_{kn}^{(k)}(x)=M_{n}^{k}(x)$.
	\end{proposition}
 \begin{proof}
 	It can be proved similar to proposition(\ref{obsr}).
 \end{proof}
	\begin{theorem}
		For $n,s \in \mathbb{N}$, we have 
		$$M_{sn+1}^{(s)}(x) =3xM_{sn}^{(s)}(x)-2M_{sn-1}^{(s)}(x).$$
	\end{theorem}
 \begin{proof} 
 	 From eqn.(\ref{polyrelation}) and eqn.(\ref{mersenepolynomial}), we have
 	\begin{eqnarray}
 		3xM_{sn}^{(s)}(x)-2M_{sn-1}^{(s)}(x) &=& 3xM_{n}^{s}(x)-2M_{n-1}(x)M_{n}^{s-1}(x) \nonumber\\
 		&=& M_{n}^{s-1}(x)(3xM_{n}(x)-2M_{n-1}(x))\nonumber\\
 		&=& M_{n}^{s-1}(x)M_{n+1}(x) \nonumber\\
 		&=& M_{sn+1}^{(s)}(x). \nonumber
 	\end{eqnarray}
 \end{proof}
 \begin{theorem}
 	For $k,s \in \mathbb{N}$ we have, $$M_{s+1}^{k}(x)-M_{s}^{k}(x)=M_{sk+k}^{(k)}(x)-M_{sk}^{(k)}(x).$$
 \end{theorem}
 \begin{proof}
 	From eqn.(\ref{polyrelation}), we have
 	\begin{eqnarray}
 		M_{sk+k}^{(k)}(x)-M_{sk}^{(k)}(x) &=& [M_{s}^{k-k}(x)M_{s+1}^{k}(x)]-[M_{s}^{k-0}(x)M_{s+1}^{0}(x)] \nonumber\\
 		&=& M_{s+1}^{k}(x)-M_{s}^{k}(x). \nonumber 
 	\end{eqnarray}
 \end{proof}
 \begin{theorem}
 	Let $n,k \ge 2$, we can write Cassini's identity for $M_{n}^{(k)}(x)$ as\\
 	$$M_{nk+a}^{(k)}(x)M_{nk+a-2}^{(k)}(x)-(M_{nk+a-1}^{(k)})^{2}(x)=
 	\begin{cases}
 		-2^{n-1}M_{n}^{2k-2}(x),~~~~ a = 1 \\
 		0,\hspace{2.9cm} a \ne 1
 	\end{cases}$$
 \end{theorem}
 \begin{proof}
 	Let $a\ne 1$, so from eqn.(\ref{relation}) we have
 	\begin{eqnarray}
 		M_{nk+a}^{(k)}(x)M_{nk+a-2}^{(k)}(x)-(M_{nk+a-1}^{(k)})^{2}(x) 
 		&=&(M_{n}^{k-a}(x)M_{n+1}^{a}(x))(M_{n}^{k-a+2}(x)M_{n+1}^{a-2}(x))\nonumber\\
 		& & -(M_{n}^{k-a+1}(x)M_{n+1}^{a-1})^{2}(x)\nonumber\\
 		&=& M_{n}^{2k-2a+2}(x)[M_{n+1}(x)M_{2a-2}(x)-(M_{n+1})^{2a-2}(x)].\nonumber\\
 		&=& 0 \nonumber
 	\end{eqnarray}
 	and if $a=1$, 
 	\begin{eqnarray}
 		M_{nk+1}^{(k)}(x)M_{nk-1}^{(k)}(x)-(M_{nk}^{(k)})^{2}(x) &=& (M_{n}^{k-1}(x)M_{n+1}(x))(M_{n-1}M_{n}^{k-1}(x))-(M_{n}^{k})^{2}(x)\nonumber\\
 		&=& M_{n}^{2k-2}(x)[M_{n+1}(x)M_{n-1}(x)-(M_{n})^{2}(x)]\nonumber\\
 		&=& -2^{n-1}M_{n}^{2k-2}(x). \hspace{.5in} \text{(Using theorem \ref{Cassinipolymial})} \nonumber
 	\end{eqnarray}	 
 \end{proof} 
	\section{Generalized Gaussian Mersenne numbers and their polynomials }
	\subsection{Gaussian Mersenne Numbers}
	\begin{definition}\label{gaussmersene}
		The \textbf{Gaussian Mersenne sequence$\{GM_{k}\}$} is defined by the recurrence relation,
		\begin{eqnarray}
			GM_{k+2} = 3GM_{k+1} -2GM_{k}, \hspace{.8cm}k\geq 0
		\end{eqnarray}
		with $GM_{0}=-i/2,~~GM_{1}=1$.
	\end{definition}
	The first few Gaussian Mersenne numbers are $-i/2,1,3+i,7+3i,15+7i,...$
	\\Further, the relation between Gaussian Mersenne numbers and classical Mersenne numbers is $$GM_{k+2} = M_{k+2}+iM_{k+1},$$ where $GM_{k}$ is the $k^{th}$-Gaussian Mersenne number.
	\begin{theorem}
		For every $n \in \mathbb{N}$, Binet formula for the Gaussian Mersenne numbers is 
		\begin{equation}\label{GaussBinet}
			GM_{n}=(\lambda_{1}^{n}-\lambda_{2}^{n})+i(\lambda_{1}^{n-1}-\lambda_{2}^{n-1}),
		\end{equation}
		where $\lambda_{1}$ and $\lambda_{2}$ are the roots of the characteristic equation(\ref{char}).
	\end{theorem}
	\begin{proof}
		The $n^{th}$ term of Gaussian Mersenne numbers for the difference equation(\ref{gaussmersene}) is,
		\begin{equation}\label{general}
			GM_{n}= a\lambda_{1}^{n}+b\lambda_{2}^{n}. 
		\end{equation}
	To eliminate constants $a$ $\&$ $b$, we use initial conditions given in eqn.(\ref{gaussmersene}). \\Since, we have $GM_{0}= a+b$ and $GM_{1}= a\lambda_{1}+b\lambda_{2}$. \\It yields, $a=1+i/2$ and $b=-1-i$.
		\\So, from eqn.(\ref{general}) we get 
		$$GM_{n}=(\lambda_{1}^{n}-\lambda_{2}^{n})+i(\lambda_{1}^{n-1}-\lambda_{2}^{n-1}).$$
		Furthermore, by using values of  $\lambda_{1}$ and $\lambda_{2}$ in eqn.(\ref{GaussBinet}), we get another form of Binet formula which is 		\begin{equation}\label{GaussBinet2}
			GM_{n}=(2^{n}-1)+i(2^{n-1}-1).
		\end{equation}
	\end{proof}
	\begin{theorem}[Catlan's Identity]
		For $n,m \geq 1$, we have
		\begin{equation}\label{gaussCatlan}
			GM_{n+m}GM_{n-m}-GM_{n}^{2}=\left[(2^{n}-2^{n+m-1})+(2^{n-m-1}-2^{n-m})\right]+i3(2^{n}-2^{n+m-1}-2^{n-m-1}).
		\end{equation}
	\end{theorem}
	 \begin{proof}
	  Using the Binet formula(\ref{GaussBinet2}), we have
	  \begin{eqnarray}
	  	GM_{n+m}GM_{n-m}-GM_{n}^{2}&=&\left[(2^{n+m}-1)+i(2^{n+m-1}-1)(2^{n+m}-1)+i(2^{n+m-1}-1)\right]\nonumber\\
	  	& & -\left[(2^{n}-1)+i(2^{n-1}-1)\right]^{2}\nonumber\\
	  &=&\left[(2^{n}-2^{n+m-1})+(2^{n-m-1}-2^{n-m})\right]+i3(2^{n}-2^{n+m-1}-2^{n-m-1}).\nonumber
	  \end{eqnarray}
	 \end{proof}
	 \textbf{Note:} If $m=1$ in the above Catlan's identity(\ref{gaussCatlan}), we get Cassini's identity for the Gaussian Mersenne numbers and hence the following result.
	 \begin{theorem}[Cassini's Identity]
	 	For $n \geq 1$,
	 	\begin{equation}\label{gaussCassini}
	 		GM_{n+1}GM_{n-1}-GM_{n}^{2}=(2^{n-2}-2^{n-1})-i(3.2^{n-2}).
	 	\end{equation}
	 \end{theorem}
	  \begin{theorem}[D'Ocagne's Identity]
	 	For $n,m \geq 1$,
	 	\begin{equation}\label{gaussOcagne}
	 		GM_{m+1}GM_{n}-GM_{m}GM_{n+1}=(2^{n-1}-2^{m-1})+i3(2^{n-1}-2^{m-1}).
	 	\end{equation}
	 \end{theorem}
	\begin{proof}
		From eqn.(\ref{GaussBinet2}), we have
		\begin{eqnarray}
			GM_{m+1}GM_{n}-GM_{m}GM_{n+1} &=&\left[ (2^{m+1}-1)+i(2^{m}-1)(2^{n}-1)+i(2^{n-1}-1)\right]\nonumber\\
			& & -\left[\ (2^{m}-1)+i(2^{m-1}-1)(2^{n+1}-1)+i(2^{n}-1)\right]\nonumber\\
			&=&(2^{n-1}-2^{m-1})+i3(2^{n-1}-2^{m-1}).\nonumber
		\end{eqnarray}
	\end{proof}
	\begin{theorem}[Generating function]
	Generating function for Gaussian Mersenne numbers is
	\begin{equation*}\label{gaussgenerating}				GM(z)=\dfrac{z+i\left(\dfrac{3}{2}z-\dfrac{1}{2}\right)}{(1-3z+2z^{2})}.
	\end{equation*}
	\end{theorem}
	\begin{proof}
		The Generating function for the sequence $\{GM_{n}\}_{n \in \mathbb{N}}$ is given by
		$$GM(z)=\sum_{j=0}^{\infty}GM_{j}z^{j}.$$
		i.e.
		$$GM(z)=GM_{0}+GM_{1}z^{1}+GM_{2}z^{2}+...+GM_{n}z^{n}+...$$
		Now,
		\begin{eqnarray}\label{gen}
			GM(z)-3zGM(z)+2z^{2}GM(z)=GM_{0}+z(GM_{1}-3GM_{0})
		\end{eqnarray}
		So, using initial values(\ref{gaussmersene}) in eqn.(\ref{gen}), we have
		\begin{equation}
			GM(z)=\dfrac{z+i\left(\dfrac{3}{2}z-\dfrac{1}{2}\right)}{(1-3z+2z^{2})}.
		\end{equation}
	\end{proof}

	\subsection{Generalized $k$-Gaussian Mersenne numbers}
	\begin{definition}\label{Gaussl-mersenne}
		Let $n,k\in \mathbb{N}$ then $\exists $ $s,r$ such that $n=sk+r$, $0 \le r <k$. The generalized $k$-Gaussian Mersenne numbers $(GM_{n}^{(k)})$ are defined by 
		\begin{equation}
			GM_{n}^{(k)}=\left( \lambda_{1}^{s}-\lambda_{2}^{s}+i(\lambda_{1}^{s-1}-\lambda_{2}^{s-1})\right)^{k-r}\left( \lambda_{1}^{s+1}-\lambda_{2}^{s+1}+i(\lambda_{1}^{s}-\lambda_{2}^{s})\right)^{r},~~~~~n=sk+r
		\end{equation}
		where $\lambda_{1}$ and $\lambda_{2}$ are the roots of the characteristic equation (\ref{char}).
	\end{definition}
	The relation between generalized $k$-Gaussian Mersenne numbers and Gaussian Mersenne numbers are (See the eqn.(\ref{GaussBinet}) and definition(\ref{Gaussl-mersenne})) 
	\begin{equation}\label{Gaussrelation}
		GM_{n}^{(k)}=GM_{s}^{k-r}GM_{s+1}^{r},~~~n=sk+r.
	\end{equation}
	If $k=1$ then $r=0$ and hence $m=n$. So, from eqn.(\ref{Gaussrelation}) we have $GM_{n}^{(1)}= GM_{n}$.
	\\In particular for $k=2,3,4$, the relation between generalized $k$-Gaussian Mersenne numbers and Gaussian Mersenne numbers are as following: 
	\begin{enumerate}
		\item $ GM_{2n}^{(2)}=GM_{n}^{2}$.
		\item $ GM_{2n+1}^{(2)}=GM_{n}GM_{n+1}$.
		\item $ GM_{2n+1}^{(2)}=3GM_{2n}^{(2)}-2GM_{2n-1}^{(2)}$.
		\item $ GM_{3n}^{(3)}=GM_{n}^{3}$.
		\item $ GM_{3n+1}^{(3)}=GM_{n}^{2}GM_{n+1}$.
		\item $ GM_{3n+1}^{(3)}=3GM_{3n}^{(3)}-2GM_{3n-1}^{(3)}$.
		\item $ GM_{3n+2}^{(3)}=GM_{n}GM_{n+1}^{2}$.	
	\end{enumerate}
		Some generalized $k$-Gaussian Mersenne numbers are listed in the following table.
	\begin{table}[h!]
		\begin{center}
			\begin{tabular}{c c c c c c c}
				\hline 
				$GM_{n}^{(k)}$ & k=1 & k=2  &  k=3 &  k=4 & k=5   \\
				\hline
				$GM_{0}^{(k)}$ & -i/2 & -1/4  &  i/8 &  1/16 & -i/32   \\
				
				$GM_{1}^{(k)}$  &  1 & -i/2  &  -1/4 &  i/8 & 1/16   \\
				
				$GM_{2}^{(k)}$ & 3+i & 1  &  -i/2 &  -1/4 & i/8   \\
				
				$GM_{3}^{(k)}$& 7+3i & 3+i  &  1 &  -i/2 & -i/4  \\
				
				$GM_{4}^{(k)}$ & 15+7i & 8+6i  &  3+i &  1 & -i/2  \\
				
				$GM_{5}^{(k)}$& 31+15i & 18+16i  &  8+6i  &  3+i &  1  \\
				\hline	
			\end{tabular}
			\caption{List of some generalized $k$-Gaussian Mersenne numbers $GM_{n}^{(k)}$} 
		\end{center}
	\end{table}
 
	\begin{proposition}\label{obsr2}
		Let $k,n \in \mathbb{N}$, then we have $ GM_{nk}^{(k)}=GM_{n}^{k}$.
	\end{proposition}
	\begin{theorem}
		For $n, s \in \mathbb{N}$, we have $GM_{sn+1}^{(s)} = 3GM_{sn}^{(s)}-2GM_{sn-1}^{(s)}$.
	\end{theorem}
	\begin{proof}
		From eqn.(\ref{Gaussrelation}), we get
		\begin{eqnarray}
			3GM_{sn}^{(s)}-2GM_{sn-1}^{(s)} &=& 3GM_{n}^{s}-2GM_{n-1}GM_{n}^{s-1} \nonumber\\
			&=& GM_{n}^{s-1}GM_{n+1}\nonumber\\
			&=& GM_{sn+1}^{(s)}.\nonumber
		\end{eqnarray}
	\end{proof}
	\begin{theorem}
		For $k,s \in \mathbb{N}$, we have $GM_{s+1}^{k}-GM_{s}^{k}=GM_{sk+k}^{(k)}-GM_{sk}^{(k)}$.
	\end{theorem}
	\begin{proof}
		From eqn.(\ref{Gaussrelation}), we obtain
		\begin{eqnarray}
			GM_{sk+k}^{(k)}-GM_{sk}^{(k)} &=& [GM_{s}^{k-k}GM_{s+1}^{k}]-[GM_{s}^{k-0}GM_{s+1}^{0}] \nonumber\\
			&=& GM_{s+1}^{k}-GM_{s}^{k}. \nonumber 
		\end{eqnarray}
	\end{proof}
	\begin{theorem}
		Let $n,m \ge 0$ such that $n+m>1$, then we have
		\begin{equation}
			GM_{2(n+m-1)}^{(2)}-GM_{n+m}GM_{n+m-2}=(2^{n+m-1}-2^{n+m-2})+i(3.2^{n+m-2})\nonumber
		\end{equation}
	\end{theorem}
	\begin{proof}
		By eqn.(\ref{gaussCassini}) and proposition(\ref{obsr2}), we get
		\begin{equation}
			GM_{2(n+m-1)}^{(2)}-GM_{n+m}GM_{n+m-2}=GM_{(n+m-1)}^{2}-GM_{n+m}M_{n+m-2}=(2^{n+m-1}-2^{n+m-2})+i(3.2^{n+m-2}).\nonumber
		\end{equation}
	\end{proof}
	\begin{theorem}
		Let $n,k \ge 2$, then Cassini's identity for $GM_{n}^{(k)}$ is,
		$$GM_{nk+a-1}^{2}-GM_{nk+a}^{(k)}GM_{nk+a-2}^{(k)}=
		\begin{cases}
			GM_{n}^{2k-2}\left[(2^{n-1}-2^{n-2})+i(3.2^{n-2})\right],~~~~ a = 1 \\
			0, ~~~~ a \ne 1\\
		\end{cases}$$.
	\end{theorem}
	\begin{proof}
		If $a=1$, so from eqn.(\ref{Gaussrelation})
		\begin{eqnarray}
			GM_{nk+1}^{(k)}GM_{nk-1}^{(k)}-(GM_{nk}^{(k)})^{2} &=& (GM_{n}^{k-1}GM_{n+1})(GM_{n-1}GM_{n}^{k-1})-(GM_{n}^{k})^{2} \nonumber\\
			&=& GM_{n}^{2k-2}[GM_{n+1}GM_{n-1}-(GM_{n})^{2}]\nonumber\\
			&=& GM_{n}^{2k-2}\left[(2^{n-1}-2^{n-2})+i(3.2^{n-2})\right].\nonumber \hspace{.5in} \text{(Using eqn.(\ref{gaussCassini}))}\nonumber
		\end{eqnarray}	 
		and if $a\ne 1$, then by eqn.(\ref{Gaussrelation})
		\begin{eqnarray}
			GM_{nk+a}^{(k)}GM_{nk+a-2}^{(k)}-(GM_{nk+a-1}^{(k)})^{2}  &=&  (GM_{n}^{k-a}
			GM_{n+1}^{a})(GM_{n}^{k-a+2}GM_{n+1}^{a-2})-(GM_{n}^{k-a+1}GM_{n+1}^{a-1})^{2}\nonumber\\
			&=& GM_{n}^{2k-2a+2}[GM_{n+1}^{2a-2}-(GM_{n+1})^{2a-2}].\nonumber\\
			&=& 0.
		\end{eqnarray}
	\end{proof}
	\subsection{Gaussian Mersenne Polynomials}
	\begin{definition}
		The \textbf{Gaussian Mersenne polynomials($GM_{k}(x)$)}  are defined by the recurrence relation,
		\begin{eqnarray}\label{gaussmersenepoly}
			GM_{k+2}(x) = 3xGM_{k+1}(x) -2GM_{k}(x), \hspace{.8cm}k\geq 0
		\end{eqnarray}
		with $GM_{0}=-i/2,~~GM_{1}=1$.
	\end{definition}
	The first few Gaussian Mersenne polynomials are $-i/2,~ 1,~ 9x^{2}-2+i3x,~ 27x^{3}-12x+i(9x^{2}-2).$
	\\Further, the Gaussian Mersenne polynomials and Mersenne polynomials are related as $$GM_{k+2} =M_{k+2}(x)+iM_{k+1}(x),$$ where $M_{k}(x)$ is the $k^{th}$-Mersenne polynomial.
	\begin{theorem}
		For every $n \in \mathbb{N}$, the Binet formula for the Gaussian Mersenne polynomials is
		\begin{equation}\label{GaussBinetpoly}
			GM_{n}(x)=\left(\dfrac{\lambda_{1}^{n}(x)-\lambda_{2}^{n}(x)}{\lambda_{1}(x)-\lambda_{2}(x)}\right)+i\left(\dfrac{\lambda_{1}^{n-1}(x)-\lambda_{2}^{n-1}(x)}{\lambda_{1}(x)-\lambda_{2}(x)}\right),
		\end{equation}
		where $\lambda_{1}(x)$ and $\lambda_{2}(x)$ are the roots of the characteristic equation(\ref{charpoly}).
	\end{theorem}
	\begin{proof}
		The general term of Gaussian Mersenne polynomials for the recurrence relation(\ref{gaussmersenepoly}) can be obtained by,
		\begin{equation}\label{generalgausspoly}
			GM_{n}(x)= a\lambda_{1}^{n}(x)+b\lambda_{2}^{n}(x). 
		\end{equation}
		Using initial conditions given in eqn.(\ref{gaussmersenepoly}), we eliminate constants $a$ $\&$ $b$.
		\\Since, we have $GM_{0}= a+b$ and $GM_{1}= a\lambda_{1}(x)+b\lambda_{2}(x)$.
		\\It yields, $a=1+i/2$ and $b=-1-i$.
		\\Thus, from eqn.(\ref{generalgausspoly}), we have 
		$$GM_{n}(x)=\left(\dfrac{\lambda_{1}^{n}(x)-\lambda_{2}^{n}(x)}{\lambda_{1}(x)-\lambda_{2}(x)} \right)+i\left(\dfrac{\lambda_{1}^{n-1}(x)-\lambda_{2}^{n-1}(x)}{\lambda_{1}(x)-\lambda_{2}(x)}\right).$$
	\end{proof}
\begin{theorem}[Cassini's Identity]
	For $n \geq 1$, we have
	\begin{equation}\label{gaussCassinipoly}
		GM_{n+1}(x)GM_{n-1}(x)-GM_{n}^{2}(x)=(2^{n-2}-2^{n-1})-i(x3.2^{n-2}).
	\end{equation}
\end{theorem}
 \begin{proof}
 	We prove it using mathematical induction on $n$. For $n=1$, 
 	\begin{equation}
 		GM_{2}(x)GM_{0}(x)-GM_{1}^{2}(x)=(3x+i)(-i/2)-1=(2^{-1}-1)-i(x3.2^{-1}).\nonumber
 	\end{equation}  
 	So, the statement is true for $n=1$.
 	\\Assume that result is true for $n=k$, $i.e.$ 
 	\begin{equation}\label{Cassinipoly3}
 		GM_{k+1}(x)GM_{k-1}(x)-GM_{k}^{2}(x)=(2^{k-2}-2^{k-1})-i(x3.2^{k-2}). 
 	\end{equation}
 	Now, for $n=k+1$, using eqn.(\ref{gaussmersenepoly}) and eqn.(\ref{Cassinipoly3}), we have
 	\begin{eqnarray}
 		GM_{k+2}(x)GM_{k}(x)-GM_{k+1}^{2}(x)
 		&=&\left[\left(3xGM_{k+1}(x)-2GM_{k}(x)\right)\left(\frac{1}{3x}GM_{k+1}(x)+\frac{2}{3x}GM_{k-1}(x)\right)\right]\nonumber\\
 		& & -GM_{k+1}^{2}(x) \hspace{1in} \text{(Using eqn.(\ref{gaussmersenepoly}))}\nonumber\\
 		&=& 2GM_{k+1}(x)GM_{k-1}(x)-\frac{2}{3x}GM_{k+1}(x)GM_{k}(x)\nonumber\\
 		& & -\frac{4}{3x}GM_{k}(x)GM_{k-1}(x)\nonumber\\
 		&=& 2GM_{k}^{2}(x)+2.(2^{k-2}-2^{k-1})-i(x3.2^{k-2})-\frac{2}{3x}GM_{k+1}(x)GM_{k}(x)\nonumber\\
 		& & +\frac{4}{3x}GM_{k-1}(x)GM_{k}(x)\nonumber\\
 		&=& (2^{k-1}-2^{k})-i(x3.2^{k-1}).\nonumber
 	\end{eqnarray}
 \end{proof}

	\subsection{ Generalized $k$-Gaussian Mersenne polynomials}
	\begin{definition}\label{Gaussl-mersennepoly}
		Let $n,k\in \mathbb{N}$, then $\exists$ $s,r$ such that $n=sk+r$, $0 \le r <k$. The generalized $k$-Gaussian Mersenne polynomials $(GM_{n}^{(k)}(x))$ are defined by 
		\begin{equation}
			GM_{n}^{(k)}(x)=\left(\dfrac{ (\lambda_{1}^{s}-\lambda_{2}^{s})+i(\lambda_{1}^{s-1}-\lambda_{2}^{s-1})}{\lambda_{1}(x)-\lambda_{2}(x)}\right)^{k-r}\left( \dfrac{(\lambda_{1}^{s+1}-\lambda_{2}^{s+1})+i(\lambda_{1}^{s}-\lambda_{2}^{s})}{\lambda_{1}(x)-\lambda_{2}(x)} \right)^{r},~~~~~n=sk+r
		\end{equation}
		where $\lambda_{1}(x)$ and $\lambda_{2}(x)$ are the roots of the characteristic equation (\ref{charpoly}).
	\end{definition}
	The relation between the generalized $k$-Gaussian Mersenne polynomials and Gaussian Mersenne polynomials are 
	\begin{equation}\label{Gaussrelationpoly}
		GM_{n}^{(k)}(x)=GM_{s}^{k-r}(x)GM_{s+1}^{r}(x),~~~n=sk+r.
	\end{equation}
	If $k=1$ then $r=0$ and hence $m=n$. So, from eqn.(\ref{Gaussrelationpoly}) we have $GM_{n}^{(1)}(x)= GM_{n}(x)$.
	\\
	 The generalized $k$-Gaussian Mersenne polynomials and Gaussian Mersenne polynomials holds the following relations for $k=2,3$. 
	 \begin{enumerate}
	 	\item $ GM_{2n}^{(2)}(x)=GM_{n}^{2}(x)$.
	 	\item $ GM_{2n+1}^{(2)}(x)=GM_{n}(x)GM_{n+1}(x)$.
	 	\item $ GM_{2n+1}^{(2)}(x)=3xGM_{2n}^{(2)}(x)-2GM_{2n-1}^{(2)}(x)$.
	 	\item $ GM_{3n}^{(3)}(x)=GM_{n}^{3}(x)$.
	 	\item $ GM_{3n+1}^{(3)}(x)=GM_{n}^{2}(x)GM_{n+1}(x)$.
	 	\item $ GM_{3n+1}^{(3)}(x)=3xGM_{3n}^{(3)}(x)-2GM_{3n-1}^{(3)}(x)$.
	 	\item $ GM_{3n+2}^{(3)}(x)=GM_{n}(x)GM_{n+1}^{2}(x)$.	
	\end{enumerate}
	Following table shows the list of some generalized $k$-Gaussian Mersenne polynomials $(GM_{n}^{(k)}(x))$.
	\begin{table}[h] 
		\centering	
 		\begin{tabular}{c c c c c c c}
 			\hline 
 			$GM_{n}^{(k)}(x)$ & k=1 & k=2  &  k=3 &  k=4 & k=5   \\
 			\hline
 			$GM_{0}^{(k)}(x)$ & -i/2 & -1/4  &  i/8 &  1/16 & -i/32   \\
 			
 			$GM_{1}^{(k)}(x)$  &  1 & -i/2  &  -1/4 &  i/8 & 1/16   \\
 			
 			$GM_{2}^{(k)}(x)$ & 3$x$+i & 1  &  -i/2 &  -1/4 & i/8   \\
 			
 			$GM_{3}^{(k)}(x)$& 9$x^{2}$-2+i3$x$,  & 3$x$+i  &  1 &  -i/2 & -i/4  \\
 			
 			$GM_{4}^{(k)}(x)$ & 27$x^{3}$-12x+i(9$x^{2}$-2) & (9$x^{2}$-1)+i6$x$  &  3$x$+i &  1 & -i/2  \\
 			
 			$GM_{5}^{k}(x)$& (81$x^{4}$-54$x^{2}$+4)+i(27$x^{3}$-12x) &(27$x^{3}$-3$x^{2}$-6$x$)+i(18$x^{2}$-2)  &  (9$x^{2}$-1)+i6$x$  &  3$x$+i &  1  \\
 			\hline	
 		\end{tabular}
 	\caption{List of some generalized $k$-Gaussian Mersenne polynomials $GM_{n}^{(k)}(x)$} 
 \end{table}
 
 \begin{proposition}\label{obsr3}
 	Let $k,n \in \mathbb{N}$, then we have $ GM_{kn}^{(k)}(x)=GM_{n}^{k}(x)$.
 \end{proposition}
 \begin{theorem}
 	For $n, s \in \mathbb{N}$, we have $ GM_{sn+1}^{(s)}(x)= 3xGM_{sn}^{(s)}(x)-2GM_{sn-1}^{(s)}(x)$.
 \end{theorem}
 \begin{proof}
 	From equation(\ref{Gaussrelationpoly}),
 	\begin{eqnarray}
 		3xGM_{sn}^{(s)}(x)-2GM_{sn-1}^{(s)}(x) &=& 3xGM_{n}^{s}(x)-2GM_{n-1}(x)GM_{n}^{s-1}(x)\nonumber\\
 		&=& GM_{n}^{s-1}(x)GM_{n+1}(x)\nonumber\\
 		&=& GM_{sn+1}^{(s)}(x).\nonumber
 	\end{eqnarray}
 \end{proof}
 \begin{theorem}
 	For $k,s \in \mathbb{N}$, we have $GM_{s+1}^{k}(x)-GM_{s}^{k}(x)=GM_{sk+k}^{(k)}(x)-GM_{sk}^{(k)}(x)$.
 \end{theorem}
 \begin{proof}
 	From equation(\ref{Gaussrelationpoly}), we obtain
 	\begin{eqnarray}
 		GM_{sk+k}^{(k)}(x)-GM_{sk}^{(k)}(x) &=& [GM_{s}^{k-k}(x)GM_{s+1}^{k}(x)]-[GM_{s}^{k-0}(x)GM_{s+1}^{0}(x)] \nonumber\\
 		&=& GM_{s+1}^{k}(x)-GM_{s}^{k}(x). \nonumber 
 	\end{eqnarray}
 \end{proof}
 \begin{theorem}
 	Let $n,m \ge 0$ such that $n+m>1$, then we have
 	\begin{equation}
 		GM_{2(n+m-1)}^{(2)}(x)-GM_{n+m}(x)GM_{n+m-2}(x)=(2^{n+m-1}-2^{n+m-2})+i(x3.2^{n+m-2})\nonumber
 	\end{equation}
 \end{theorem}
 \begin{proof}
 	By eqn.(\ref{gaussCassinipoly}) and proposition(\ref{obsr3}), we get
 	\begin{eqnarray}
 		GM_{2(n+m-1)}^{(2)}(x)-GM_{n+m}(x)GM_{n+m-2}(x)
 		&=& GM_{(n+m-1)}^{2}(x)-GM_{n+m}(x)GM_{n+m-2}(x)\nonumber\\
 		&=&(2^{n+m-1}-2^{n+m-2})+i(x3.2^{n+m-2}).\nonumber
 	\end{eqnarray}
 \end{proof}
 \begin{theorem}
 	Let $n,k \ge 2$, then we can write Cassini's identity for $GM_{n}^{(k)}(x)$ as,
 	$$GM_{nk+a-1}^{2}(x)-GM_{nk+a}^{(k)}(x)GM_{nk+a-2}^{(k)}(x)=
 	\begin{cases}
 		GM_{n}^{2k-2}(x)\left[(2^{n-1}-2^{n-2})+i(x3.2^{n-2})\right],~~~~ a = 1 \\
 		0, \hspace{2.8in} a \ne 1\\
 	\end{cases}$$.
 \end{theorem}
 \begin{proof}
 	Let $a=1$, then using eqn.(\ref{Gaussrelationpoly}) we have,
 	\begin{eqnarray}
 		GM_{nk+1}^{(k)}(x)GM_{nk-1}^{(k)}(x)-(GM_{nk}^{(k)})^{2}(x)
 		&=&(GM_{n}^{k-1}(x)GM_{n+1}(x))(GM_{n-1}(x)GM_{n}^{k-1}(x))-(GM_{n}^{k})^{2}(x) \nonumber\\
 		&=& GM_{n}^{2k-2}(x)[GM_{n+1}(x)GM_{n-1}(x)-(GM_{n})^{2}(x)]\nonumber\\
 		&=& GM_{n}^{2k-2}(x)(2^{n-1}-2^{n-2})+i(x3.2^{n-2}).\nonumber \hspace{.5in} \text{(Using eqn.\ref{gaussCassinipoly})}\nonumber
 	\end{eqnarray}	 
 	and if $a \ne 1$, then by eqn.(\ref{Gaussrelationpoly}),
 	\begin{eqnarray}
 		GM_{nk+a}^{(k)}(x)GM_{nk+a-2}^{(k)}(x)-(GM_{nk+a-1}^{(k)})^{2}(x)
 		&=&  (GM_{n}^{k-a}(x)GM_{n+1}^{a}(x))(GM_{n}^{k-a+2}(x)GM_{n+1}^{a-2}(x)) \nonumber\\
 		& &	-(GM_{n}^{k-a+1}(x)GM_{n+1}^{a-1}(x))^{2}\nonumber\\
 		&=& GM_{n}^{2k-2a+2}(x)[GM_{n+1}^{2a-2}(x)-(GM_{n+1})^{2a-2}(x)] \nonumber\\
 		&=& 0.\nonumber
 	\end{eqnarray}
 \end{proof}
	\subsection*{Declaration of competing interest}
	The authors declare that they have no known competing financial interests or personal relationships that could have appeared to influence the work reported in this paper.
	\subsection*{Acknowledgment}      
	The authors are thankful to anonymous reviewer for their care and advice. The first \& third author acknowledge the University Grant Commission(UGC), India for providing fellowship for this research work.
	
	\bibliography{Recurrsive}
	\bibliographystyle{acm}
\end{document}